\documentclass[twoside]{article}
\usepackage{graphicx,amssymb,mathrsfs,amsmath,enumerate}
\usepackage{amsthm}
\usepackage{charter}
\catcode`@=11
\long\def\@makefntext#1{\noindent #1}
\newskip\tabcentering \tabcentering=1000pt plus 1000pt minus 1000pt
\def\REF#1{\par\hangindent\parindent\indent\llap{#1\enspace}\ignorespaces} 
\def\MCH#1#2{\setbox0=\hbox{\raise#1\hbox{#2}}\smash{\box0}}

\def\@evenfoot{}\def\@oddfoot{}

\def\@evenhead{\hbox to\textwidth{\footnotesize\rm\thepage \hfill
{\it Cheng Zhiyun, Gao Hongzhu}}} 

\def\@oddhead{\hbox to \textwidth{\footnotesize{\it
On Region Crossing Change and Incidence Matrix} \hfill\thepage}}

\catcode`@=12

\def\bc{\begin{center}}
\def\ec{\end{center}}
\def\no{\noindent}
\def\hang{\hangindent\parindent}
\def\textindent#1{\indent\llap{\qquad #1\ \ \enspace}\ignorespaces}
\def\ref{\par\hang\textindent}

\newtheorem{theorem}{Theorem}[section]

\newtheorem{lemma}[theorem]{Lemma}
\newtheorem{corollary}[theorem]{Corollary}

\newtheorem{proposition}[theorem]{Proposition}

\begin{document}
\abovedisplayskip=6pt plus 1pt minus 1pt \belowdisplayskip=6pt
plus 1pt minus 1pt
\thispagestyle{empty} \vspace*{-1.0truecm} \noindent
\vskip 10mm

\bc{\large\bf On Region Crossing Change and Incidence Matrix\\[2mm]
\footnotetext{\footnotesize The authors are supported by NSF 10671018 and Scientific Research Foundation of Beijng Normal University}} \ec

\vskip 5mm
\bc{\bf Cheng Zhiyun\, Gao Hongzhu\\
{\small School of Mathematical Sciences, Beijing Normal University
\\Laboratory of Mathematics and Complex Systems, Ministry of
Education, Beijing 100875, China
\\(email: czy@mail.bnu.edu.cn; hzgao@bnu.edu.cn)}}\ec

\vskip 1 mm

\noindent{\small {\small\bf Abstract} In a recent work of Ayaka Shimizu$^{[5]}$, she defined an operation named region crossing change on link diagrams, and showed that region crossing change is an unknotting operation for knot diagrams. In this paper, we prove that region crossing change on a 2-component link diagram is an unknotting operation if and only if the linking number of the diagram is even.

Besides, we define an incidence matrix of a link diagram via its signed planar graph and its dual graph. By studying the relation between region crossing change and incidence matrix, we prove that a signed planar graph represents an $n$-component link diagram if and only if the rank of the associated incidence matrix equals to $c-n+1$, here $c$ denotes the size of the graph.
\ \

\vspace{1mm}\baselineskip 12pt

\no{\small\bf Keywords} region crossing change; incidence matrix \ \

\no{\small\bf MR(2000) Subject Classification} 57M25 05C50\ \ {\rm }}

\vskip 1 mm


\vspace{1mm}\baselineskip 12pt



\section{Introduction}
In knot theory, unknotting operation is an interesting and important research topic. Generally speaking, an \textit{unknotting operation} is a local move on knot diagrams such that one can deform any knot into a trivial knot by some such local moves. The ordinary unknotting operation is crossing change. In [2], Hitoshi Murakami defined $\sharp$-operation and proved that $\sharp$-operation is an unknotting operation. Later, $\triangle$-operation was defined in [3] and it was proved that $\triangle$-operation is also an unknotting operation. In 1990, Hoste, Nakanishi and Taniyama defined $H(2)$-move in [1] and proved that $H(2)$-move is a kind of unknotting operation.
\begin{center}
\includegraphics{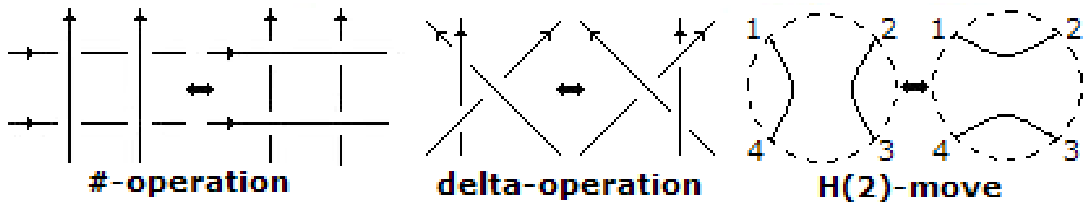} \centerline{\small Figure
1\quad}
\end{center}

Recently, Ayaka Shimizu posted a paper in which a new operation on link diagrams called region crossing change was defined. Here a \textit{region crossing change} at a region of $R^2$ divided by a link diagram is defined to be the crossing changes at all the crossing points on the boundary of the region. The figure below shows the effect of region crossing change on the gray region:
\begin{center}
\includegraphics{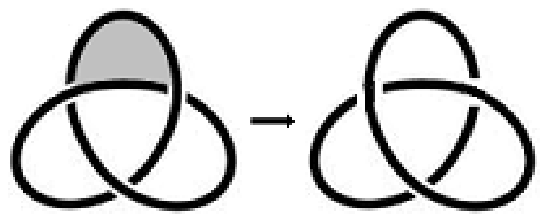} \centerline{\small Figure
2\quad}
\end{center}
The main result in [5] is that region crossing change on knot diagram is an unknotting operation. In fact, the author proved that:
\begin{theorem}$^{[5]}$Let $D$ be a knot diagram and $P$ a crossing point of $D$, then there exist region crossing changes which transform $D$ into a new knot diagram $D'$, here $D'$ is obtained from $D$ by a crossing change at $P$.
\end{theorem}

Similar to [5], given a link diagram $D$, and a region $R$ of $D$, we use $D(R)$ to denote the new diagram obtained from $D$ by a region crossing change at $R$. Obviously $(D(R_1))(R_2)=(D(R_2))(R_1)$ and $(D(R_1))(R_1)=D$. Hence $D(R_1\cup R_2)$ makes sense, and from now on we simply write $D(R_1R_2)$ for $D(R_1\cup R_2)$. If we use $R_1, \cdots, R_n$ to denote all the regions of $D$, obviously there are $2^n$ different cases of regions crossing changes totally. The theorem above can be described as for any crossing point $P$ there exists a subset $\alpha \subset \{1, \cdots, n\}$ such that $D(R_\alpha)$ is obtained from $D$ by a crossing change at $P$, here $D(R_\alpha)$ denotes $D(\bigcup\limits_{i\in \alpha} R_i)$.

In [5], the author gave an example of the standard diagram of Hopf link to explain that in general region crossing change is not an unknotting operation for links. An interesting question is that when region crossing change is an unknotting operation for link diagrams. In the first part of this paper, we will give an answer to this question for 2-component link diagrams. We prove that:
\begin{theorem}Given a 2-component link $L=K_1\cup K_2$, let $D$ be a diagram of $L$, region crossing change on $D$ is an unknotting operation if and only if $lk(K_1, K_2)$ is even.
\end{theorem}

For links with more than two components, we have a sufficient condition as follows:
\begin{theorem}Given an $n$-component link $L=K_1\cup \cdots \cup K_n$, if $lk(K_i, K_j)$ are all even for $1\leq i< j\leq n$, then region crossing change is an unknotting operation for any diagram of $L$.
\end{theorem}

The second part of this paper concerns the relation between region crossing change and incidence matrix. Before giving the definition of incidence matrix, we first give a brief review of the transformation between a link diagram and a signed planar graph. Let $D$ be an oriented link diagram, color the regions of $R^2$ divided by $D$ in checkerboard fashion. Since $D$ can be regarded as a 4-valent graph if we consider each crossing point as a vertex of degree 4, the color mentioned above must exist. Without loss of generality we assume the unbounded region has the white color, then we assign a vertex to every black region, an edge to every crossing. The sign of an edge is defined as follows:
\begin{center}
\includegraphics{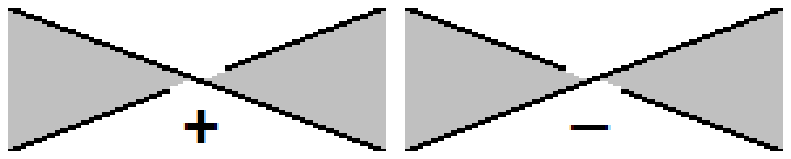} \centerline{\small Figure
3\quad}
\end{center}

From a link diagram $D$ we can obtain a signed planar graph $G$. If the crossing number of $D$ is $c$, then the size of $G$ is $c$, i.e. $G$ contains $c$ edges. Consider the dual graph of $G$, say $G'$, since $D$ is connected, then $G$ and $G'$ are both connected. It is evident the size of $G'$ is also $c$, and the the order (the number of the vertices) of $G$ plus the order of $G'$ is exactly the number of regions of $D$, which is $c+2$. In graph theory$^{[6]}$, the \textit{incidence matrix} $M(G)$ of a (undirected) graph $G$ is a $v\times e$ matrix, here $v, e$ denote the order and size of $G$ respectively:
\begin{center}
$M(G)=(m_x(y)),\quad x\in V(G)$ and $y\in E(G)$
\end{center}
and
\begin{center}
$m_x(y)=
\begin{cases}
1& \text{if $y$ is incident with $x$;}\\
0& \text{otherwise.}
\end{cases}$
\end{center}

Let $W$ and $B$ be the number of white regions and black regions of $D$ respectively. Then $M(G)$ is a $B\times c$ matrix and $M(G')$ is a $W\times c$ matrix. With the 1-1 correspondence between the edge set of $G$ and $G'$, we can construct a new $(W+B)\times c=(c+2)\times c$ matrix $M(D)$ as below:
\begin{center}
$M(D)=\begin{bmatrix}
 M(G)     \\
 M(G')  \\
\end{bmatrix}$
\end{center}

We call this $(c+2)\times c$ matrix $M(D)$ the \textit{incidence matrix} of the diagram $D$. We remark that $M(D)$ is not well-defined unless we fix an order of vertex sets of $G$ and $G'$, and an order of the edge set of $G$. However the rank of $M(D)$ is independent of those order mentioned above, hence it is well-defined. Since one vertex of $G$ or $G'$ corresponds to a region of $D$, and an edge of $G$ corresponds to a crossing point of $D$, hence we can name the regions and crossing points of $D$ by $R_1, \cdots, R_{c+2}$ and $P_1, \cdots, P_c$, respectively. Now it is easy to find that an element $m_{ij}$ of $M(D)$ is 1 or 0 exactly corresponding to whether $P_j$ is on the boundary of $R_i$ or not.

In general, given a signed planar graph $G$, it is not easy to detect the number of components of the link that $G$ represents. Sometimes we especially concern whether a signed planar graph corresponds to a knot diagram. For example, it is not evident that the graph below corresponds to a multi-component link diagram rather than a knot diagram. Since the sign is not important here, so there is no need for us to mention it.
\begin{center}
\includegraphics{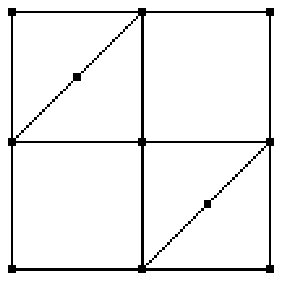} \centerline{\small Figure
4\quad}
\end{center}
One motivation of this paper is to give a complete solution to the question above. In fact, we prove that:
\begin{theorem}
A signed planar graph $G$ represents an $n$-component link diagram if and only if the $\mathbb{Z}_2$-rank of $M(D)$ equals to $c-n+1$, here $c$ denotes the size of $G$.
\end{theorem}

We want to fix two conventions we will use throughout. First, a diagram always means a non-split diagram, i.e. if we regard the diagram as a 4-valent planar graph, then it is connected. Second, we will work with $\mathbb{Z}_2$ coefficients, i.e. linearly (in)dependence always means $\mathbb{Z}_2$-linearly (in)dependence, and the rank of a matrix always means $\mathbb{Z}_2$-rank of the matrix. Note that $\mathbb{Z}_2$-linearly independence induces the linearly independence with coefficient $\mathbb{Z}$. The remainder of the paper is organized as follows: in section 2 we will give the proof of Theorem 1.2 and Theorem 1.3. Section 3 contains the proof of a special case of Theorem 1.4, the knot diagram. Some relations between region crossing changes and incidence matrix are also discussed. Finally, in Section 4 we give the proof of Theorem 1.4.

\section{Region crossing change on 2-component link diagram}
In this section we will study the behavior of region crossing changes on 2-component link diagram $D$. Before giving the proof of Theorem 1.2, we first give a proposition as follows:
\begin{proposition}
Given a 2-component link diagram $D=K_1\cup K_2$, suppose $D'$ is obtained from $D$ by a crossing change at a crossing point of $K_1\cap K_1$ or $K_2\cap K_2$, and $D''$ is obtained from $D$ by crossing changes at two crossing points of $K_1\cap K_2$. Then both $D'$ and $D''$ can be obtained from $D$ by region crossing changes.
\end{proposition}
\begin{proof}
The fact $D'$ can be obtained from $D$ by region crossing changes mainly follows from Theorem 1.1. Let $P$ be a crossing point of $K_1\cap K_1$ and $R_\alpha$ those regions of $K_1$ corresponding to Theorem 1.1. Note that each region of $K_1$ is the union of some regions of $K_1\cup K_2$, and if each region of $R_\alpha$ includes no nugatory crossing of $K_2\cap K_2$, then $R_\alpha$ satisfy our requirement.

If one region $R_i (i\in \alpha)$ includes some nugatory crossing points of $K_2\cap K_2$, let us consider the three regions around a nugatory crossing point. Given a nugatory crossing point, there exist some regions which can be contained in a disk whose boundary intersects the diagram only at the nugatory crossing. We call these regions a \textit{reducible part} of the nugatory crossing, name the region on the other side of the nugatory crossing the \textit{opposite region}, for the third region, we use the \textit{outer region} to denote it. See the figure below:
\begin{center}
\includegraphics{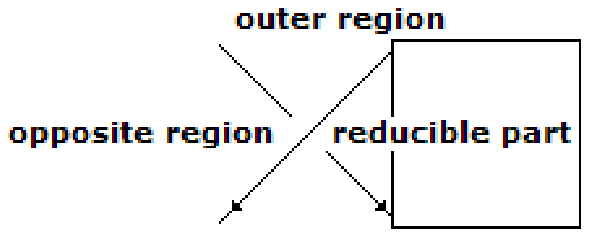} \centerline{\small Figure
5\quad}
\end{center}

Now we want to find some regions such that the effect of region crossing changes on these regions will change the crossing points of $K_1\cap K_1$ on the boundary of $R_i$ but preserving the others. First we color $R_i$ black. It is obvious that given two reducible parts $T_1$, $T_2$ in $R_i$, if $T_1\cap T_2\neq \emptyset$ then either $T_1\subset T_2$ or $T_2\subset T_1$. We apply the algorithm below to recolor those reducible parts from outside to inside, i.e. if $T_1\subset T_2$, we apply the recoloring for $T_2$ before $T_1$.
\begin{itemize}
\item The opposite region and the outer region are both colored black. Then we recolor the third region around the nugatory crossing white, and recolor other regions of $T$ in checkerboard fashion, according to this white region.
\item The opposite region and the outer region are both colored white. Then we recolor all regions of $T$ white.
\item The opposite region is colored black and the outer region is colored white. Then we recolor the third region around the nugatory crossing black, and recolor other regions of $T$ in checkerboard fashion, according to this black region.
\item The opposite region is colored white and the outer region is colored black. Then we recolor all regions of $T$ black.
\end{itemize}
For example, for a reducible part which is not contained in any other reducible part, we apply the first case of the algorithm, since we color $R_i$ black first.

After recoloring all the reducible parts, we take region crossing change on all the regions with black color, it is easy to find that all the crossing points of $K_1\cap K_2$ on the boundary of $R_i$ and $K_2\cap K_2$ in the inner of $R_i$ are preserved, and all the crossing points of $K_1\cap K_1$ on the boundary of $R_i$ are changed. Repeat the process for all regions of $R_\alpha$, then we can obtain $D'$ from $D$ by region crossing changes.

Now we show that $D''$ also can be obtained from $D$ by region crossing changes. Let $P, Q$ be two crossing points of $K_1\cap K_2$. Consider crossing point $P$, we use $R_1, R_2, R_3, R_4$ to denote the regions around $P$. Since $P\in K_1\cap K_2$, then $P$ is not a nugatory crossing, hence $R_1\neq R_3$ and $R_2\neq R_4$. A crossing point can be resolved in two ways, say 0-smoothing and 1-smoothing, and both of them can make $K_1$ and $K_2$ into one component. Now we take 0-smoothing at $P$, see the figure below:
\begin{center}
\includegraphics{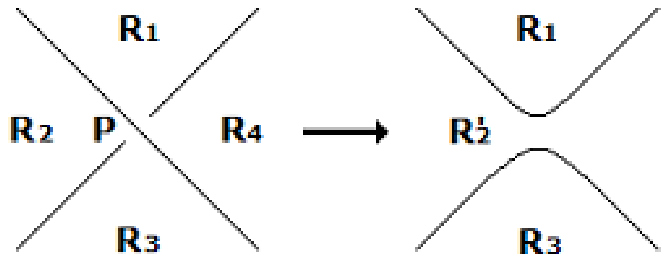} \centerline{\small Figure
6\quad}
\end{center}
After the 0-smoothing, $R_2$ and $R_4$ become one region, say $R_2'$. The link diagram is changed to be a knot diagram, and $Q$ is a crossing point of this knot diagram. According to Theorem 1.1 we can find some regions $R_\alpha$ such that taking regions crossing changes at these regions the crossing point $Q$ will be changed while others are preserved.

If $R_\alpha$ contain only one region of $R_1$ or $R_3$, in this case we say that $R_\alpha$ affect $P$, because if one takes the same region crossing changes $($ replacing $R_2'$ with $R_2\cup R_4$ if needed $)$ on $D$, then both $P$ and $Q$ are changed and other crossing points are preserved.

If $R_\alpha$ contain both $R_1$ and $R_3$ or none of them, in this case we say that $R_\alpha$ do not affect $P$, since if one takes the same crossing changes on $D$, the crossing point $P$ does not change. Now we can take 0-smoothing at $Q$, if the associated region crossing changes $R_\beta$ affect $Q$ itself, then these regions satisfy the requirement.

If $R_\alpha$ and $R_\beta$ do not affect $P$ and $Q$ respectively, then consider the regions $R_\gamma=R_\alpha\cup R_\beta$, which will change the crossing at $P$ and $Q$ but preserve other crossing points. The proof is finished.
\end{proof}

\textbf{Remark} We remark that during the process of the proof above, the case that $R_\alpha$ do not affect $P$ in fact can not happen. Since if so, $R_\alpha$ will only change the crossing at $Q$, but together with Proposition 2.1 we can conclude that one can use region crossing changes to change any one crossing point of $D$ with other crossing points preserved. This contradicts with Proposition 3.1 in Section 3.

Now we turn to the proof of Theorem 1.2.
\begin{proof}First we assume that $lk(K_1, K_2)$ is even, we claim that in this case we can deform $L$ into a trivial link by region crossing changes. Obviously there exist some crossing points such that when we make crossing change on these crossing points, $L$ becomes to be a trivial link. Without loss of generality, we use $\{P_1, \cdots, P_m\}$ to denote a group of these crossing points. Note that every time making a crossing change at a crossing point of $K_1\cap K_2$, $lk(K_1, K_2)$ will increase or decrease by one, and making a crossing change at a self-crossing point, $lk(K_1, K_2)$ is preserved. Since after making crossing change at $\{P_1, \cdots, P_m\}$, $lk(K_1, K_2)=0$, hence we conclude that $\{P_1, \cdots, P_m\}\cap \{K_1\cap K_2\}$ includes even elements. According to Proposition 2.1, there exist some regions such that taking crossing change at these regions, the crossing of $\{P_1, \cdots, P_m\}$ will be changed while other crossing points are preserved. We finish the proof of our claim.

Now we show that if $lk(K_1, K_2)$ is odd, region crossing change is not an unknotting operation. Recall the remark below Proposition 2.1, we conclude that $lk(K_1, K_2)$ mod2 is invariant under region crossing change. Hence the result follows.
\end{proof}

In general, for an $n$-component link diagram $D=K_1\cup \cdots \cup K_n$, here $K_1, \cdots, K_n$ denote the diagram of each component, similar to Proposition 2.1 we have:
\begin{proposition}
Given an $n$-component link diagram $D=K_1\cup \cdots \cup K_n$, suppose $D'$ is obtained from $D$ by a crossing change at a crossing point of $K_i\cap K_i$ $(1\leq i\leq n)$ , and $D''$ is obtained from $D$ by crossing changes at two crossing points of $K_i\cap K_j$ $(1\leq i< j\leq n)$. Then both $D'$ and $D''$ can be obtained from $D$ by region crossing changes.
\end{proposition}
\begin{proof}
The proof of the first part is similar to the proof of the first part of Proposition 2.1, and the second part follows from the first part and the proof of the second part of Proposition 2.1.
\end{proof}

Theorem 1.3 is direct follows from this proposition.

We remark that the inverse statement of Theorem 1.3 is incorrect. The figure below gives a counterexample.
\begin{center}
\includegraphics{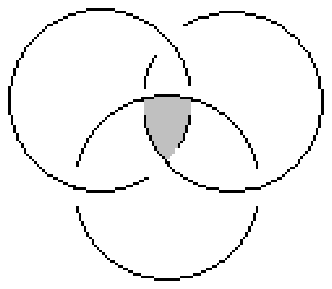} \centerline{\small Figure
7\quad}
\end{center}

\section{Incidence matrix of knot diagram}
In this section we study the relation between region crossing change and incidence matrix, and we will prove Theorem 1.4 for the case of knot diagrams.
\begin{proposition}Given a signed planar graph $G$, let $D$ denote the corresponding link diagram of $G$, then $D$ is a knot diagram if and only if the $\mathbb{Z}_2$-rank of $M(D)$ is $c$, here $c$ denotes the size of $G$.
\end{proposition}

We use $r_1, \cdots, r_{c+2}$ and $c_1, \cdots, c_c$ to denote the row vectors and column vectors of $M(D)$ respectively, which correspond to $R_1, \cdots, R_{c+2}$ and $P_1, \cdots, P_c$ as we mentioned before. Recall that an element $m_{ij}$ of $M(D)$ is 1 or 0 exactly corresponding to whether $P_j$ is on the boundary of $R_i$ or not, and note that we are working over $\mathbb{Z}_2$. Hence the positions of 1's in $r_i$ tell us which crossing points will be changed if we take the region crossing change on $R_i$. Similarly the positions of 1's in $\sum\limits_{i\in \alpha} r_i$ tell us those crossing points which will be changed if we take the region crossing changes on $R_\alpha$.

According to Theorem 1.1, for any crossing point $P_j$ of $D$, one can obtain a diagram $D'$ by region crossing changes, where $D'$ is different from $D$ at $P_j$. This means we can find some $\{r_i\}$ such that
\begin{center}
$\sum\limits_{i} r_i= (0, \cdots, 0, 1, 0, \cdots, 0)$
\end{center}
where 1 is on the $j$-th column. Since $j$ can be chosen as any number from 1 to $c$, then with $\mathbb{Z}_2$-linear combination of row vectors of $M(D)$ we can construct $c$ linearly independent row vectors. It follows that the rank of $M(D)$ is $c$, hence we have proved the necessary part of Proposition 3.1.

To prove the sufficient part of Proposition 3.1, we need a simple lemma as follows:
\begin{lemma}
For a 2-component link diagram, the boundary of each region contains even number of non-self crossing points, i.e. those crossing points which are the intersections of different components.
\end{lemma}
\begin{proof}
We can use white and black to color the diagrams of those two components. If the boundary of a region contains only self-crossing point, the lemma is correct on this region. If a region contains some non-self crossing points, then the boundary of this region can be regard as several arcs with color white, black, white, black $\cdots$ If there are $n$ white arcs, then the number of those crossing points described in the lemma is $2n$. Thus we finish the proof.
\end{proof}

Now we turn to the proof of the sufficient part of Proposition 3.1.
\begin{proof}
Given a link diagram $D$ with crossing number $c$, and the rank of $M(D)$ is $c$. We assume that $D$ is not a knot diagram, i.e. the link that associates to $D$ has at least two components.

First, we consider a simple case, we suppose the link associated to $D$ has only two components. We use $K_1$ and $K_2$ to denote the associated diagrams of these two knots, obviously $D=K_1\cup K_2$. We divide the crossing points of $D$ into three cases: $K_1\cap K_1, K_2\cap K_2$ and $K_1\cap K_2$. Without loss of generality, we assume $K_1\cap K_2=\{P_1, \cdots, P_s\}$ and $R_1, \cdots, R_t$ are those regions whose boundaries contain some crossing points of $K_1\cap K_2$. Then all elements $m_{ij}=0$ if $t+1\leq i\leq c+2$ and $1\leq j\leq s$. On the other hand, according to Lemma 3.2, $\sum\limits_{j=1}^s m_{ij}=0$ for any $1\leq i\leq t$. Therefore we conclude that $\sum\limits_{j=1}^s m_{ij}=0$ for any $1\leq i\leq c+2$, i.e.
\begin{center}
$\sum\limits_{i=1}^s c_i=\begin{bmatrix}
 0 \\
 \vdots \\
 0 \\
\end{bmatrix}.$
\end{center}
This contradicts the fact that the rank of $M(D)$ is $c$.

Now we consider the general case of $n$-component link diagram. As before we use $K_1, \cdots, K_n$ to denote the diagram of each component and $D=K_1\cup \cdots \cup K_n$. Since $D$ is non-split, it is impossible that $K_1\cap K_j=\emptyset$ for all $2\leq j\leq n$. Consider all the crossing points of $K_1\cap K_j$ where $2\leq j\leq n$, we can assume that $\bigcup\limits_{j=2}^n(K_1\cap K_j)=\{P_1, \cdots, P_u\}$, and $R_1, \cdots, R_v$ are those regions whose boundaries contain some crossing points of $K_1\cap K_j$. Similar to the proof of Lemma 3.2, it is not difficult to observe that $\sum\limits_{j=1}^u m_{ij}=0$ for any $1\leq i\leq c+2$. It follows that
\begin{center}
$\sum\limits_{i=1}^u c_i=\begin{bmatrix}
 0 \\
 \vdots \\
 0 \\
\end{bmatrix},$
\end{center}
which contradicts the assumption that the rank of $M(D)$ is $c$.
\end{proof}

If a knot diagram $D$ is reduced, i.e. $D$ contains no nugatory crossing, then both $G$ and $G'$ contain no self-loop, i.e. $G$ and $G'$ contain no such edge which connects a vertex to itself. It is obvious that the rank of $M(G)\leq B-1$ and the rank of $M(G')\leq W-1$, since the sum of all the row vectors of $M(G)$ $(M(G'))$ is 0. On the other hand the rank of $M(D)$ is $c$, hence we conclude that the rank of $M(G)$ and $M(G')$ are $B-1$ and $W-1$ respectively, and arbitrary $B-1$ $(W-1)$ row vectors of $M(G)$ $(M(G'))$ are linearly independent hence form a basis for $M(G)$ $(M(G'))$. Then the rank of $M(D)$ equals to $c$ if and only if one takes arbitrary $B-1$ row vectors of $M(G)$ and arbitrary $W-1$ row vectors of $M(G')$ together they are still linearly independent.

We remark that Proposition 3.1 gives a sufficient and necessary condition to detect whether a signed planar graph corresponds to a knot diagram. For some special cases, we have a corollary below, which implies that the planar graph in Figure 4 corresponds to a multi-component link diagram.
\begin{corollary}
If the degree of each vertex of $G$ and $G'$ is even, then $G$ corresponds to a multi-component link diagram.
\end{corollary}
\begin{proof}
According to the assumption, each row of $M(D)$ contains even 1's, hence \begin{center}
$\sum\limits_{i=1}^c c_i=\begin{bmatrix}
 0 \\
 \vdots \\
 0 \\
\end{bmatrix}.$
\end{center} By Proposition 3.1, the result follows.
\end{proof}

\textbf{Remark} In [5], the author pointed out that on the standard diagram of Hopf link one can not make crossing change at only one crossing point by region crossing changes, since two crossing points both lie on the boundaries of all four regions. From Proposition 3.1 we conclude that with region crossing changes, if we can obtain all those diagrams which are different from the original diagram at one crossing, then the diagram must be a knot diagram.

\section{Incidence matrix of $n$-component link diagram}
In this section we give the proof of Theorem 1.4.
\begin{proof}Obviously it suffices to prove that if $D$ is an $n$-component link diagram, then rank of $M(D)$ equals to $c-n+1$. For $n=1$, the statement is correct by Proposition 3.1. We assume the statement is correct for any $(n-1)$-component link diagram, we want to show that it is also correct for $n$-component link diagram $D$. As before, we write $D=K_1\cup K_2\cup\cdots \cup K_n$, here $K_1, K_2, \cdots, K_n$ denote the diagrams of the components. Let $D'$ be the diagram of $K_1\cup K_2\cup\cdots \cup K_{n-1}$, with $w$ crossing points. We use $u$ and $v$ to denote the number of crossing points of $K_n\cap K_n$ and the number of crossing points of $K_n\cap D'$ respectively. Therefore by induction the rank of $M(D')$ is $w-n+2$, we want to show the rank of $M(D)$ equals to $u+v+w-n+1$, note that $u+v+w=c$ at present.

The first step, we show that $c-n+1$ is an upper bound for the rank of $M(D)$. Let $P_{\alpha_{ij}}$ denote all the crossing points of $K_i\cap K_j$ $(1\leq i, j\leq n)$, similar to the proof of the sufficient part of Proposition 3.1, for any $1\leq i\leq n$ we have $\sum\limits_{k\in\alpha_{ij}, j\neq i} c_k=\begin{bmatrix}
 0 \\
\end{bmatrix}_{(c+2)\times 1}.$ Note that for any $1\leq i\leq n$, there exists $j\in \{1, \cdots, \hat{i}, \cdots, n\}$ such that $K_i\cap K_j\neq \emptyset$, hence $\alpha_{ij}\neq \emptyset$. It means that now we have $n$ equations. However, it is not difficult to find that the last equation $\sum\limits_{k\in\alpha_{nj}, j\neq n} c_k=\begin{bmatrix}
 0 \\
\end{bmatrix}_{(c+2)\times 1}$ can be obtained from $\sum\limits_{k\in\alpha_{ij}, j\neq i} c_k=\begin{bmatrix}
 0 \\
\end{bmatrix}_{(c+2)\times 1}$ $(1\leq i\leq n-1)$. For these $n-1$ equations, we claim that $\sum\limits_{k\in\alpha_{ij}, j\neq i} c_k$ $(1\leq i\leq n-1)$ are linearly independent, which implies that the rank of $M(D)\leq c-n+1$. Now we assume that $\sum\limits_{k\in\alpha_{ij}, j\neq i} c_k$ $(1\leq i\leq n-1)$ are linearly dependent, without loss of generality, we assume that $\sum\limits_{k\in\alpha_{tj}, j\neq t} c_k=\begin{bmatrix}
 0 \\
\end{bmatrix}_{(c+2)\times 1}$ can be obtained from $\sum\limits_{k\in\alpha_{ij}, j\neq i} c_k=\begin{bmatrix}
 0 \\
\end{bmatrix}_{(c+2)\times 1}$ $(1\leq i\leq n-1, i\neq t)$. Now we continue our discussion in two cases:
\begin{itemize}
\item If $K_t\cap K_n\neq \emptyset$, then $\alpha_{tn}\neq \emptyset$. However those columns $c_k (k\in\alpha_{tn})$ will not appear in $\sum\limits_{k\in\alpha_{ij}, j\neq i} c_k$ $(1\leq i\leq n-1, i\neq t)$. This is a contradiction.
\item If $K_t\cap K_n=\emptyset$. Without loss of generality, we assume for $1\leq i\leq s$, $K_i$ are those diagrams which satisfy $K_i\cap K_n\neq \emptyset$, and for $s+1\leq i\leq n-1$, $K_i\cap K_n=\emptyset$. Then it is easy to conclude that for any $1\leq i\leq s$, $\sum\limits_{k\in\alpha_{ij}, j\neq i} c_k$ are needless if we want to use $\sum\limits_{k\in\alpha_{ij}, j\neq i} c_k$ $(1\leq i\leq n-1, i\neq t)$ to express $\sum\limits_{k\in\alpha_{tj}, j\neq t} c_k$, since those columns $c_k$ $(k\in\alpha_{in})$ only appear in $\sum\limits_{k\in\alpha_{ij}, j\neq i} c_k$ itself. On the other hand, for the same reason, if we can obtain $\sum\limits_{k\in\alpha_{tj}, j\neq t} c_k$ from $\sum\limits_{k\in\alpha_{ij}, j\neq i} c_k$ $(s+1\leq i\leq n-1, i\neq t)$ $($without loss of generality, we suppose all these $n-s-2$ summations are necessary to express $\sum\limits_{k\in\alpha_{tj}, j\neq t} c_k$$)$, then $K_i\cap K_j=\emptyset$ if $1\leq i\leq s$ or $i=n$ and $s+1\leq j\leq n-1$. This contradicts the assumption that $D$ is non-split.
\end{itemize}

The second step, we want to use induction to show that $c-n+1$ is a lower bound for the rank of $M(D)$. Without loss of generality, we use $c_1, \cdots, c_u$ to denote those columns corresponding to $K_n\cap K_n$, use $c_{u+1}, \cdots, c_{u+v}$ to denote those columns corresponding to $K_n\cap D'$, and use $c_{u+v+1}, \cdots, c_{u+v+w}$ to denote the remainder. We divide $M(D)$ into nine submatrices as below:
\begin{center}
$\begin{bmatrix}
A_{u\times u} & B_{u\times v} & C_{u\times w} \\
D_{v\times u} & E_{v\times v} & F_{v\times w} \\
G_{(w+2)\times u} & H_{(w+2)\times v} & I_{(w+2)\times w}\\
\end{bmatrix}.$
\end{center}

Consider the $(u+v+w+2)\times w$ matrix
\begin{center}
$\begin{bmatrix}
 C_{u\times w} \\
 F_{v\times w} \\
 I_{(w+2)\times w}\\
\end{bmatrix}$.
\end{center}
Since $D=D'\cup K_n$, a region of $D'$ may be divided into several associated regions of $D$. With the viewpoint of matrix, this means we can use the row vectors of the $(u+v+w+2)\times w$ matrix above to construct a $(w+2)\times w$ matrix which is exactly $M(D')$. Because the rank of $M(D')$ is $w-n+2$, it follows that the rank of that $(u+v+w+2)\times w$ matrix is at least $w-n+2$. Hence there exists $w-n+2$ linearly independent column vectors from $c_{u+v+1}, \cdots, c_{u+v+w}$, without loss of generality, we call them $c_{u+v+1}, \cdots, c_{u+v+w-n+2}$.

According to Proposition 2.2, we can find some row vectors of $M(D)$ such that the sum of them is $(1, 0, \cdots, 0)$. Hence with some elementary row operations we can make the first row to be $(1, 0, \cdots, 0)$. This process will continue till the matrix $A_{u\times u}$ becomes an identity matrix and $B_{u\times v}=0$, $C_{u\times w}=0$. Next since $A_{u\times u}$ is an identity matrix, with some elementary row operations we can make $D_{v\times u}=0$ and $G_{(w+2)\times u}=0$.

Similarly, by Proposition 2.2, we can use some row vectors of $M(D)$ to construct a row vector $(0, \cdots, 0, 1, 0, \cdots, 0, 1, 0, \cdots, 0)$, the pair of 1's locate on the $i$-th column and the $j$-th column, where $u+1\leq i, j\leq u+v$. Since the new matrix is obtained from $M(D)$ by taking some elementary row operations, we still can find some row vectors from the new matrix to obtain the row vector above. It is evident that the first $u$ row vectors $r_1, \cdots, r_u$ are not necessary here, hence we can take some elementary row operations to make
\begin{center}
$E_{v\times v}=\begin{bmatrix}
 1 & 1 &   &  &   \\
   & 1 & 1 &  &   \\
   &   &\cdots&&  \\
   &   &  & 1 & 1 \\
 \ast  & \ast  & \cdots  & \ast  & \ast\\
\end{bmatrix}$ and $F_{v\times w}=\begin{bmatrix}
  & &   &  &   \\
   &  &  &  &   \\
   &   &&&  \\
   &   &  &  &  \\
  \ast &  \ast & \cdots  &  \ast & \ast\\
\end{bmatrix}$.
\end{center}
Hence with some more elementary row operations, we can make
\begin{center}
$E_{v\times v}=\begin{bmatrix}
 1 & 1 &   &  &   \\
   & 1 & 1 &  &   \\
   &   &\cdots&&  \\
   &   &  & 1 & 1 \\
   &   &   &   & \ast\\
\end{bmatrix}$ and $H_{(w+2)\times v}=\begin{bmatrix}
  &  &  &  & \ast \\
  &  &  &  & \ast \\
  &  &  &  & \vdots\\
  &  &  &  & \ast\\
  &  &  &  & \ast\\
\end{bmatrix}$.
\end{center}

Now consider the columns $\{c_1, \cdots, c_u, c_{u+1}, \cdots, c_{u+v-1}, c_{u+v+1}, \cdots, c_{u+v+w-n+2}\}$, obviously they are linearly independent. Hence the rank of $M(D)$ is at least $u+v+w-n+1$, since the crossing number $c=u+v+w$, therefore the rank of $M(D)\geq c-n+1$. Together with $M(D)\leq c-n+1$, we obtain the result of Theorem 1.4.
\end{proof}

\no
\vskip0.2in
\no {\bf References}
\vskip0.1in

\footnotesize
\REF{[1]}Hoste, J., Nakanishi, Y. and Taniyama, K., {\it Unknotting operations involving trivial tangles}. Osaka J. Math. 27, 555-566, 1990
\REF{[2]}H. Murakami, {\it Some metrics on classical knots}. Math. Ann. 270, 35-45, 1985
\REF{[3]}H. Murakami, Y. Nakanishi, {\it On a certain move generating link-homology}. Math. Ann. 284, 75-89, 1989
\REF{[4]}D. Rolfsen, {\it Knots and links}. Publish or Perish, Inc. 1976
\REF{[5]}Ayaka Shimizu, {\it Region crossing change is an unknotting operation}. math.GT/1011.6304v2, 2010
\REF{[6]}Junming Xu, {\it Theory and Application of Graphs}. Kluwer Academic Publishers, 2003
\end{document}